\documentclass[12pt]{article}
\usepackage{e-jc}

\pdfoutput=1

\usepackage{latexsym}
\usepackage[centertags]{amsmath}
\usepackage{amsfonts}
\usepackage{amssymb}
\usepackage{amsthm}
\usepackage{newlfont}
\usepackage{graphics}
\usepackage{color}

\usepackage[usenames,dvipsnames]{xcolor}

\usepackage[demo]{graphicx} 
\usepackage{wrapfig,floatrow}
\usepackage{lipsum} 
\usepackage[font=small,labelfont=bf,margin=5mm]{caption}

\RequirePackage{xparse, graphicx, caption, picins}
\DeclareDocumentCommand \addpic{O{0.4\textwidth} m g}{\parpic[r]{%
\begin{minipage}{#1}
    \includegraphics[width=\textwidth]{#2}%
    \IfNoValueTF{#3}{}{\captionof{figure}{\footnotesize #3}}
\end{minipage}
}}

\usepackage[colorlinks=true,citecolor=black,linkcolor=black,urlcolor=blue]{hyperref}

\newcommand{\arxiv}[1]{\href{http://arxiv.org/abs/#1}{\texttt{arXiv:#1}}}

\theoremstyle{plain}
\newtheorem{theorem}{Theorem}

\newtheorem{proposition}[theorem]{Proposition}

\theoremstyle{definition}

\theoremstyle{remark}

\newcommand{\area}{\operatorname{\texttt{area}}}
\newcommand{\dinv}{\operatorname{\texttt{dinv}}}

\font\ita=cmssi12

\title{\bf Dinv and Area}

\author{Adriano Garsia\thanks{Supported by NFS grant DMS13--62160.}\\
\small Department of Mathematics\\[-0.8ex]
\small UC San Diego\\[-0.8ex]
\small California, U.S.A.\\
\small\tt garsiaadriano@gmail.com\\
\and
Guoce Xin\thanks{Partially supported by National Natural Science Foundation of China (11171231).}\\
\small School of Mathematical Sciences\\[-0.8ex]
\small Capital Normal University\\[-0.8ex]
\small Beijing 100048, PR China\\
\small\tt guoce.xin@gmail.com
}

\date{\dateline{Sep 29, 2016}{***, 2017}}

\begin{document}

\maketitle


\begin{abstract}
We give a new combinatorial proof of the
well known result that
the dinv of an $(m,n)$-Dyck path is equal to the area of its
 sweep map image. The first proof of this  remarkable identity for co-prime $(m,n)$ is due to Loehr and Warrington. There is also a second proof (in the co-prime case) due to Gorsky and Mazin and
 a third proof due to Mazin.

\end{abstract}

\section{Introduction}
Our main goal in this paper is to obtain a simpler  proof that, under the sweep map,  the {$\dinv$} statistic of a rational Dyck path $\overline D$ becomes the $\area$  statistic of its image path $D$. The first proof of this  remarkable identity is due to   Loehr and Warrington in \cite{ref4}. There is also a second proof due to Gorsky and Mazin in \cite{ref3}, and a third proof due to Mazin \cite{Mazin dinv sweep}. See Section 3 for further explanation.

Inspired by a  recent work of \cite{ref5}, we have come to depict rational Dyck paths
in  a manner which makes the ranks of the vertices of a path   consistent
with its visual representation. This very simple change turns out to be conducive to  considerable simplifications in  proving many of the properties of rational Dyck paths. For instance, we give a geometric proof of the invertibility of the rational Sweep Map in \cite{ref2}.

 We shall always use $(m,n)$ for a co-prime pair of positive integers, South end (by letter $S$) for the starting point of a North step and West end (by letter $W$) for the starting point of an East step. This is convenient and causes no confusion because we usually talk about the starting points of these steps.

\begin{figure}[!h]
  \begin{center}
     \includegraphics[width=11cm]{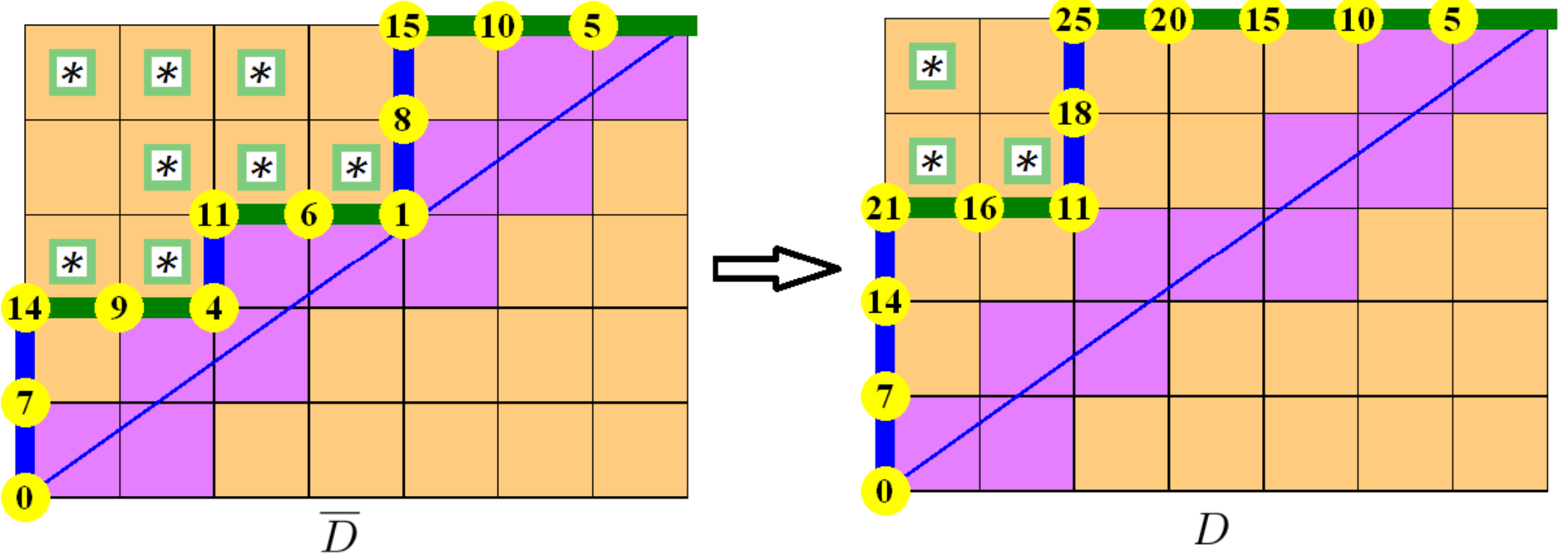}
  \end{center}
  \caption{\label{fig:GoodPair} A $(7,5)$-Dyck path and its sweep map image.}
\end{figure}

In Figure \ref{fig:GoodPair} we have depicted a path $\overline D $  in the $7\times 5$ lattice rectangle
and its sweep map image $D$ as they  are traditionally depicted. The ranks of the vertices of an $(m,n)$-path  are constructed by assigning $0$ to the origin $(0,0)$ and adding an $m$ after a North step and subtracting an $n$ after an East step.

To obtain the Sweep image $D$ of $\overline D$, we let the main diagonal (with slope $n/m$) sweep from right to left and successively draw the steps of $D$ as follows: i) draw a South end (and hence a  North step) when we sweep a South end of $\overline D $; ii) draw a West end (hence an  East step) when we sweep a West end of $\overline D$. The steps of $D$ can also be obtained by rearranging the steps of  $\overline D $ by increasing ranks of their starting  vertices.

For $(dm,dn)$-rational Dyck paths, we compute the ranks of the starting points in the same way, but we may have ties for the starting ranks. When this happens, we sweep the right starting point first. Geometrically, we may simply sweep the starting points of the steps of $\overline D$ from right to left using lines of slope $n/m+\epsilon$ for sufficiently small $\epsilon>0$, which will be written as $0<\epsilon \ll 1$.

The \texttt{area} of a rational Dyck path is equal to  the number of lattice cells between the path and   the main diagonal.
The \texttt{dinv} statistic we are using is the same as the $h^-_{m/n}$  statistic in \cite[Lemma~11]{ref1}:
A cell $c$ above a $(dm,dn)$-Dyck path $\overline D$ contributes a unit to its dinv statistic if and only if
the starting rank $a$ of the East step of $\overline D$ below $c$  and the starting rank $b$ of the North step of $\overline D$ to the right of   $c$ satisfy the inequality $1\le (b+m)-(a-n) \le m+n$, which is equivalent to
$ 0\le a-b<m+n$.
 In Figure \ref{fig:GoodPair}, the cells contributing  to
$\dinv(\overline D)$ are  distinguished by a green square\footnote{We also add a $*$ inside for black-white print.}. A quick count reveals that
$
\dinv(\overline D)=8=\area(D).
$

\begin{figure}[!h]
  \begin{center}
 \includegraphics[width=5cm]{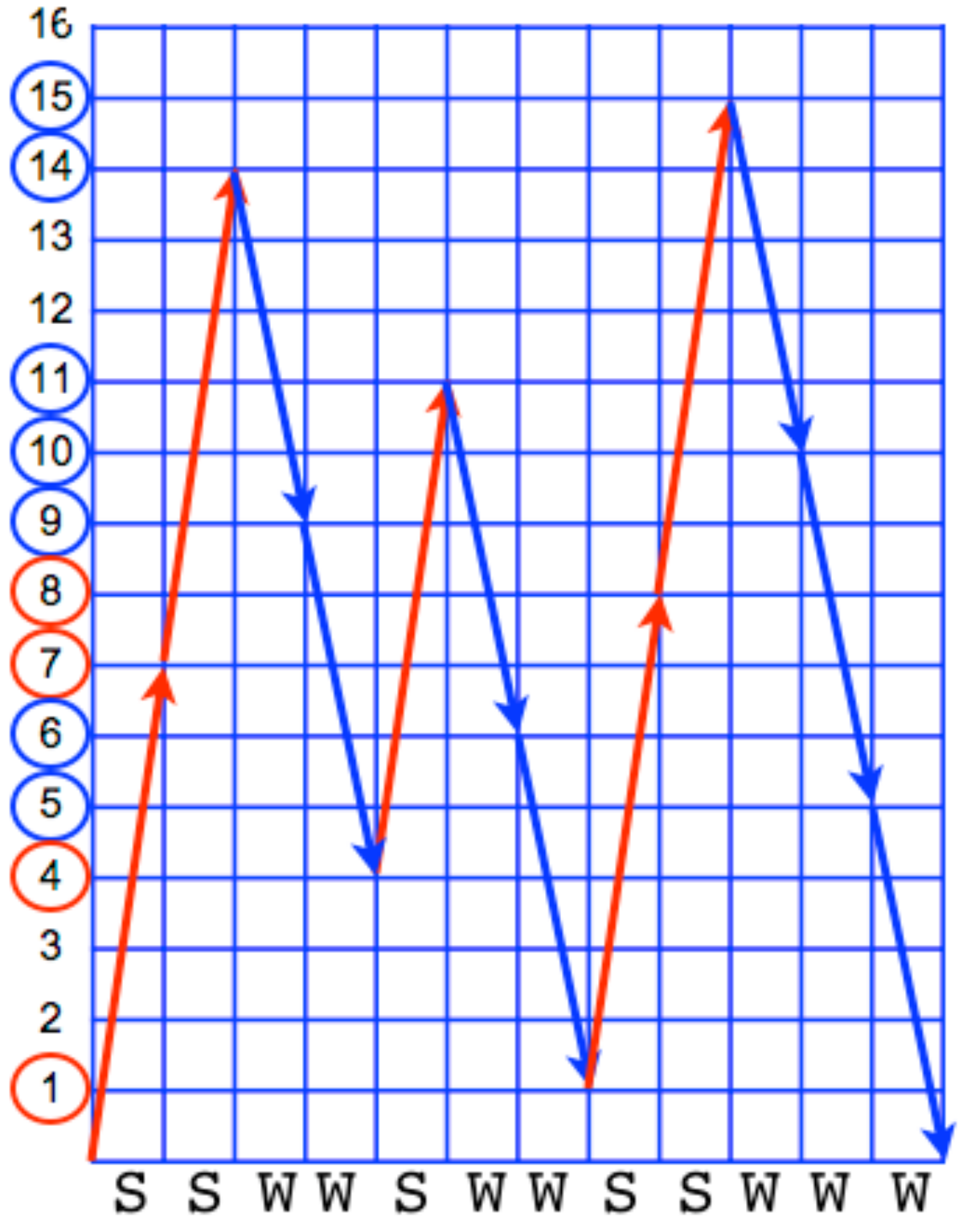}
  \end{center}
  \caption{\label{fig:Tableau} Illustration of the transformed $\overline{D}$.}
\end{figure}

To proceed  we need some notation.
A $(dm,dn)$ path diagram $T$ consists of a list of $dn$  red arrows and $dm$   blue arrows, placed on a $(dm+dn)\times dmn$ lattice rectangle.
A red arrow is the  up  vector $(1,m)$ and a blue arrow is the down vector $(1,-n)$\footnote{For black-white print, red arrows are up arrows and blue arrows are down arrows.}. The rows of lattice cells will be  referred to as {\ita rows} and  the
horizontal lattice lines  will be simply referred to as {\ita lines}. On the left of each line we have placed its $y$ coordinate which we will  refer to as its {\ita level}. The level of the starting point of an arrow is called its {\ita starting rank}, and similarly its {\ita end rank} is the level of its end point.
It will be convenient to call {\ita row $i$} the row of lattice cells delimited by the lines at levels $i$ and $i+1$.
Let $\Sigma$ be a list consisting of $dn$ letters $S$ and $dm$ letters $W$, and let $R=(r_1,\dots, r_{dm+dn})$ be a sequence of $dn+dm$ non-negative  ranks. The path diagram $T(\Sigma,R)$ (see  Figure \ref{fig:Tableau}) is obtained by placing the letters of $\Sigma$ at the bottom of the lattice columns and
if the $i^{th}$ letter of $\Sigma$ is an $S$ (resp., $W$) then we draw
a red (resp., blue) arrow with starting rank $r_i$ in the  $i^{th}$ column. Figure \ref{fig:Tableau} depicts our manner of drawing  the path $\overline D$.
The ranks  of $\overline D$ are now the circled levels of the starting vertices.

Now the sweep order is from bottom to top and from right to left within each level. Geometrically, the Sweep lines are of slope $\epsilon$ for $0<\epsilon \ll 1$. Note that the co-prime case (i.e., when $d=1$) simplifies since no two starting points have the same rank, and thus the sweep lines are just the level lines.

Notice that each lattice
cell may contain a segment of a red arrow or  a segment of a blue arrow or no segment at all. The red segment count of row $j$ will be denoted $c^r(j)$ and the blue segment count
is denoted $c^b(j)$. We will denote by $c(j)=c^r(j)-c^b(j)$ and refer to it the $j$-th row count. Observe that in every row of a path diagram, the red segments and blue segments have to alternate. In particular, for Dyck paths as in Figure \ref{fig:Tableau}, every row must start with a red segment and end with a blue segment, and hence $c(j)=0$ holds for all $j$. This is called the zero-row-count property. It has  the following immediate consequence.

\begin{proposition}\label{p-1.1}
The starting rank of  any  arrow $A$ of $D$ may be simply obtained by drawing in green (thick) the
line of slope $0<\epsilon\ll 1$ at the starting point of its preimage $\overline A$, then counting the  segments above the green line of any red arrow of $\overline D$ that starts below the  green  line and adding to that count the number of segments below the green line of any blue arrow of $\overline D$ that starts above the green line.
\end{proposition}
\begin{proof} \begin{figure}[!ht]
\centering{
\mbox{\includegraphics[height=3.5 in]{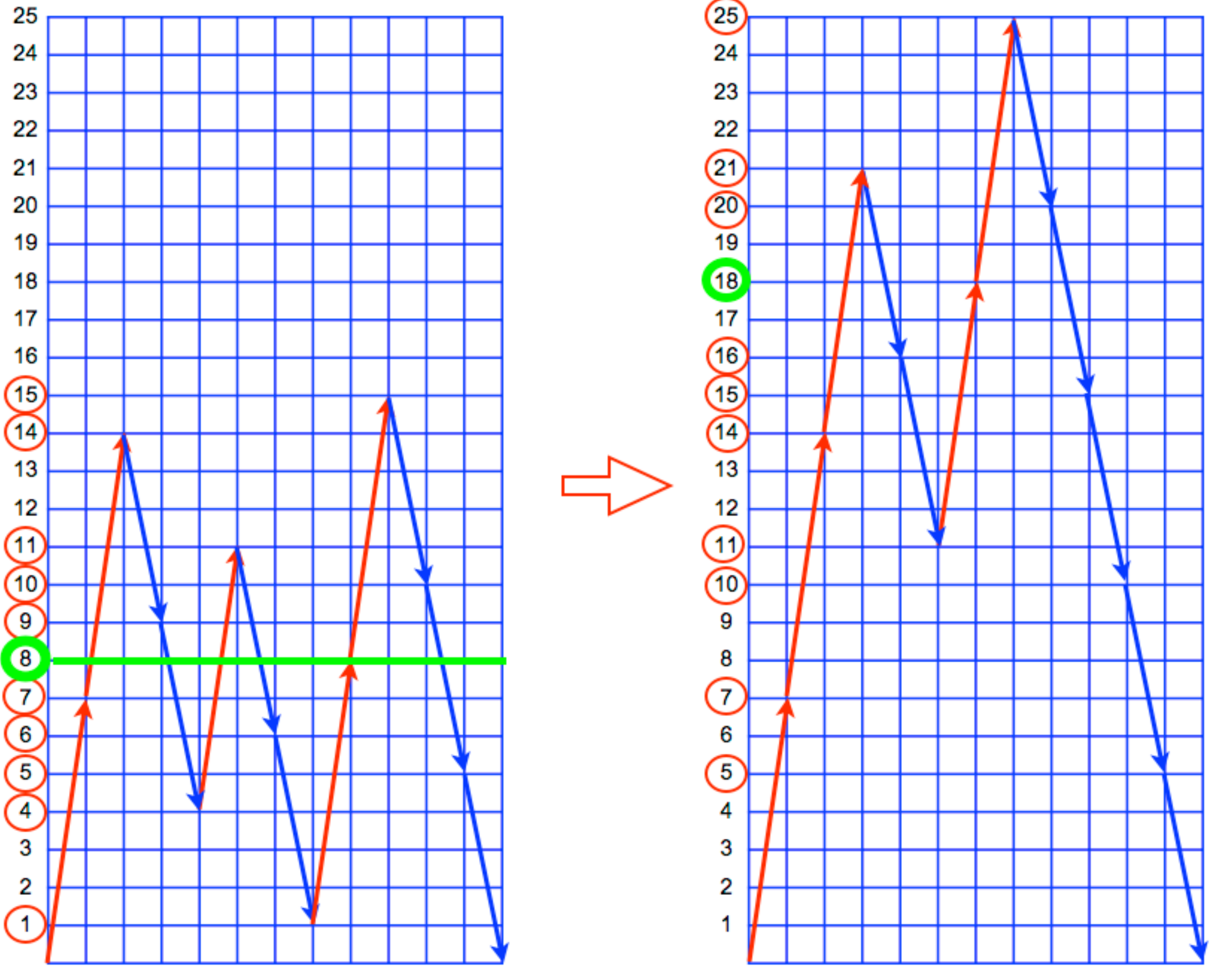}}
\caption{ Here we see that the rank of the last red arrow of $D$ is $18$.  We have   altogether $4$ red arrows and $2$ blue arrows of $\overline D$ that start below the green line, giving  $18=4\times 7 -
 2\times 5$. We count a total of $9$ red segments above the green line and $9$ blue segments below the green line. The zero-row-count property gives that these $9$ blue segments are precisely the
 needed $4\times 7 -
 2\times 5$ minus $9$.
}
\label{fig:BigPair}
}
\end{figure}

The desired rank is $bm-an$, when $\overline D$
has $b$ red arrows and $a$ blue arrows that start  below the green line. We interpret this number as the segment count of these arrows and apply the  zero-row-count property for rows below the green line. Indeed,  the count of red segments above the green line is certainly needed. All the
remaining portion of $bm-an$ are segments below the green line. On the other hand, we need to add some blue segments to have all segments below the green line to apply the zero-row-count property. These blue segments are exactly as described in the proposition.
See Figure \ref{fig:BigPair} for an example.
\end{proof}

This proposition leads to the basic formula we will use to compute the {area} of  the image $D$ by working on the preimage $\overline D$.

\begin{theorem}
  \label{t-1.1}
For any $(dm,dn)$-Dyck path $D$, with $(m,n)$ co-prime, we have
\begin{align}
\area (D) =  {1\over n}\Big(\sum_{j=1}^{dn} r(S_j(D)) \Big)-
d{n-1\over 2},
\label{e-1.1}
\end{align}
where ``$\, r(S_j(D))$'' denotes the rank of the $j^{th}$ South end of $D$.
\end{theorem}

\begin{proof}
Notice that if we define the rank  of the lattice cell whose South-West coordinates are $(i,j) $ by setting
$r(i,j)=mj-ni$, then the least positive rank in the $j^{th}$ row is none other than the remainder of $mj$ mod $n$. Since the residues  modulo $n$ are $0,1,\ldots , n-1$ and the least positive ranks are distinct for $0\le j\le n-1$, it follows that  their sum is $n(n-1)/2$. Summing over all $j$ gives $dn(n-1)/2$.
Calling $lpr_j$ the least nonnegative rank in row $j$,  it is evident that
$( r(S_j(D))-lpr_j)/n$ is the contribution of row $j$ to the {area}  of $D$.
This given, we see that \eqref{e-1.1} is simply obtained  by summing all these contributions.
\end{proof}

\section{Proof that  dinv  sweeps to  {area}}

Given a co-prime pair $(m,n)$ and $d$, our argument is to show that as we decrease the {area} of the preimage $\overline D$ by one unit, both dinv and area satisfy the same recursion.

We have depicted in Figure \ref{fig:RedBlueB} the cell that we have subtracted  from the preimage $\overline D$ to obtain
a preimage $\overline D'$ with one less unit of {area}.
This operation may be viewed as replacing a red arrow $S$ of $\overline D'$ by a dashed red arrow $S'$ and a blue arrow $W$ by   a dashed blue arrow $W'$. Calling $D$ and $D'$ the sweep map images of  $\overline D$
and  $\overline D'$, our task is to determine the difference  $\area(D)-\area(D')$. Our tool  will be formula \eqref{e-1.1} and the fact that
the starting rank of any arrow $A$ of $D$ is $bm -a n$ where
$\overline D$ has $b$  red arrows and $a$ blue arrows  that start
 below the preimage $ \overline A$ of $A$ in $\overline D$.
This difference will be calculated in $\overline D$ and
$\overline D'$ by means of our tool.  It should be mentioned that the following argument will be significantly simplified if we choose the starting rank of the displayed $W$ to be the largest in the sweep order.
Then some of the cases can not happen. In particular, the regions $T_1$ and $T_2$ can not have any segments.

Below we will talk about four lines of levels $l+n, l, k, k-n$ respectively. In the $d\ne 1$ case, we actually mean the lines of slope $\epsilon$ that passing through the corresponding four vertices of middle parallelogram. This is due to the modification of the Sweep lines in the non-coprime case.

\addpic[5cm]{{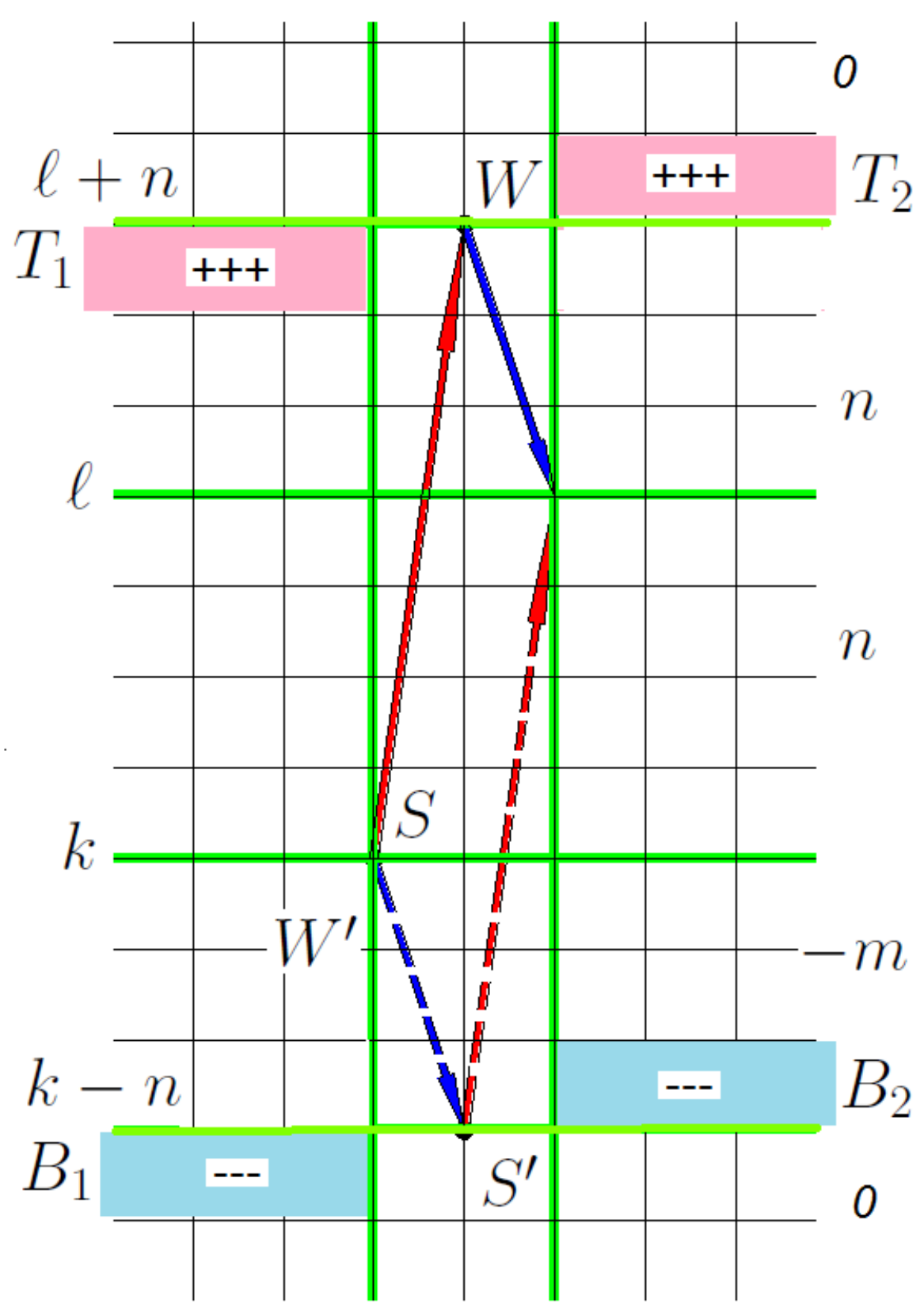}}{Contribution for the {area} difference. \label{fig:RedBlueB}}
Now there are $4$ distinct cases.   Firstly, red arrows that  start above level $\l+n$ (see Figure \ref{fig:RedBlueB}) or below level $k-n$ are not affected by the replacements.
 Thus their contribution to the difference is $0$. Secondly in the case of any  red arrow   starting strictly below level $l+n$ and strictly above level $k$, its contribution to the difference is $n$. The reason for this is that both red arrows in the display will increase the ranks of the arrows of $D$ and $D'$, whose starting levels are in this range by an equal amount, therefore they cancel each other. By contrast  the dashed blue arrow  will affect   the ranks of the red arrows of  $D'$ so that the  contribution of each  to the difference is $-(-n)$.

\noindent
Thirdly, notice that each  red arrow of $D$ or $D'$ that starts strictly below rank $k$  and  strictly above rank $k-n$, is not affected by either of the  blue arrows. But
each  of the red arrows of $D'$    is affected by the dashed   red arrow and thus contributes a $-m$ to the difference.

Finally we must include the contribution to the {areas} of $D$ and $D'$ by the ranks of red arrows in the display itself. We claim  that
\begin{multline}
  rank(S)-rank(S')\\
=m\times \#\{
\hbox{
red arrows that   start (strictly) below level $k$ and above level $k-n$
}
\}\\
-n\times \#\{
\hbox{
blue arrows that   start (strictly) below level $k$ and   above level $k-n$
}
\}.
\label{e-1.2}
\end{multline}

The reason is that the  arrows that start below
level $k-n$ contribute equally to the {areas} of $S$ and $S'$. Thus they cancel in computing the  desired difference. On the other hand all the arrows accounted for in \eqref{e-1.2} do contribute to the {area} of $S$ but not to the {area} of $S'$.

Notice first that the contribution to the {area} difference
in the third case is the negative of the first part of the contribution obtained in \eqref{e-1.2}.
After cancellation, all the remaining contributions  are multiples of $n$. (A convenient fact  since, according to \eqref{e-1.1}, $n$ has to be divided out.)

Furthermore, in the second case this multiple counts the number of red arrows that have a red segment
in $T_1$ or $T_2$. Finally, we see in \eqref{e-1.2} that the factor of $n$ that survives the cancellation, counts exactly the  arrows that have a blue segment in $B_1$ or $B_2$.

These observations imply the following result.

\begin{proposition}\label{p-2.1}
Let $\overline{D'}$ be obtained from $\overline D$
by removing an area cell and let $D$ and $D'$ be their sweep map images. Let $B_1$ and $B_2$ be the blue regions and $T_1$ and $T_2$ be the red regions\;\footnote{For black-white print, we add ``$+++$" to red regions and ``$---$" to blue regions.}
in Figure \ref{fig:RedBlueB}. Then
\begin{align}
  \area(D)-\area(D')=c^r(T_1)+c^r(T_2)-c^b(B_1)- c^b(B_2),
\label{e-1.3}
\end{align}
where $c^r(T_1),c^r(T_2)$ denote red  segment counts
and $c^b(B_1), c^b(B_2)$ denote blue  segment counts
in the corresponding regions.
\end{proposition}

To obtain the recursion satisfied by   dinv
we will make use of the following remarkable fact. By abuse of notation, we will use $W_i$ (resp. $S_j$) for the $i$-th blue ($j$-th red) arrow.

\begin{proposition}
  \label{p-2.2}
The dinv statistic of a rational  Dyck path $D$ (given in our stretched  form) may simply be obtained by counting  the pairs $(W_i\to S_j)$ consisting of a blue
 arrow to the left of a red arrow, such that   $W_i$ sweeps $S_j$, i.e., $W_i$ intersects $S_j$ when we move it along a line of slope $\epsilon$ (with $0<\epsilon \ll 1$)
 to the right   past $S_j$.
\end{proposition}
\begin{proof}
Suppose that $W_i$ has starting rank $a$ and $S_j$ has starting rank $b$. By the quoted result of Loehr-Warrington, the pair $(W_i,S_j)$ contribute a unit to
$\dinv(D)$ if and only if $0\le a-b<m+n$. This is equivalent to requiring that  $a\ge b$ and $a-n< b+m$ . But these two inequalities are  precisely what is needed to guarantee that $W_i$ sweeps $S_j$.
\end{proof}

This given, our dinv recursion can be stated as follows.

\begin{proposition}\label{p-2.3}
Let $\overline D'$ be obtained from $\overline D$
by removing an area cell.  Let $B_1$ and $T_1$ be the blue regions and $B_2$ and $T_2$ be the red regions
in Figure \ref{fig:RedBlueA}. Then
\begin{align}
\dinv(\overline D)-\dinv(\overline{D'})=c^b(T_1)+c^r(T_2)-c^b(B_1)- c^r(B_2)-1,
\label{e-1.4}
\end{align}
where $c^b(T_1),c^b(B_1)$ denote blue  segment counts
and $c^r(T_2), c^r(B_2)$ denote red  segment counts
in the corresponding regions.
\end{proposition}

\begin{proof}
We will use  the visual fact
 $
\dinv(\overline D)=\#\{(W_i\rightarrow S_j): \hbox{$W_i$ sweeps $S_j$}\}.
$

\addpic[5cm]{{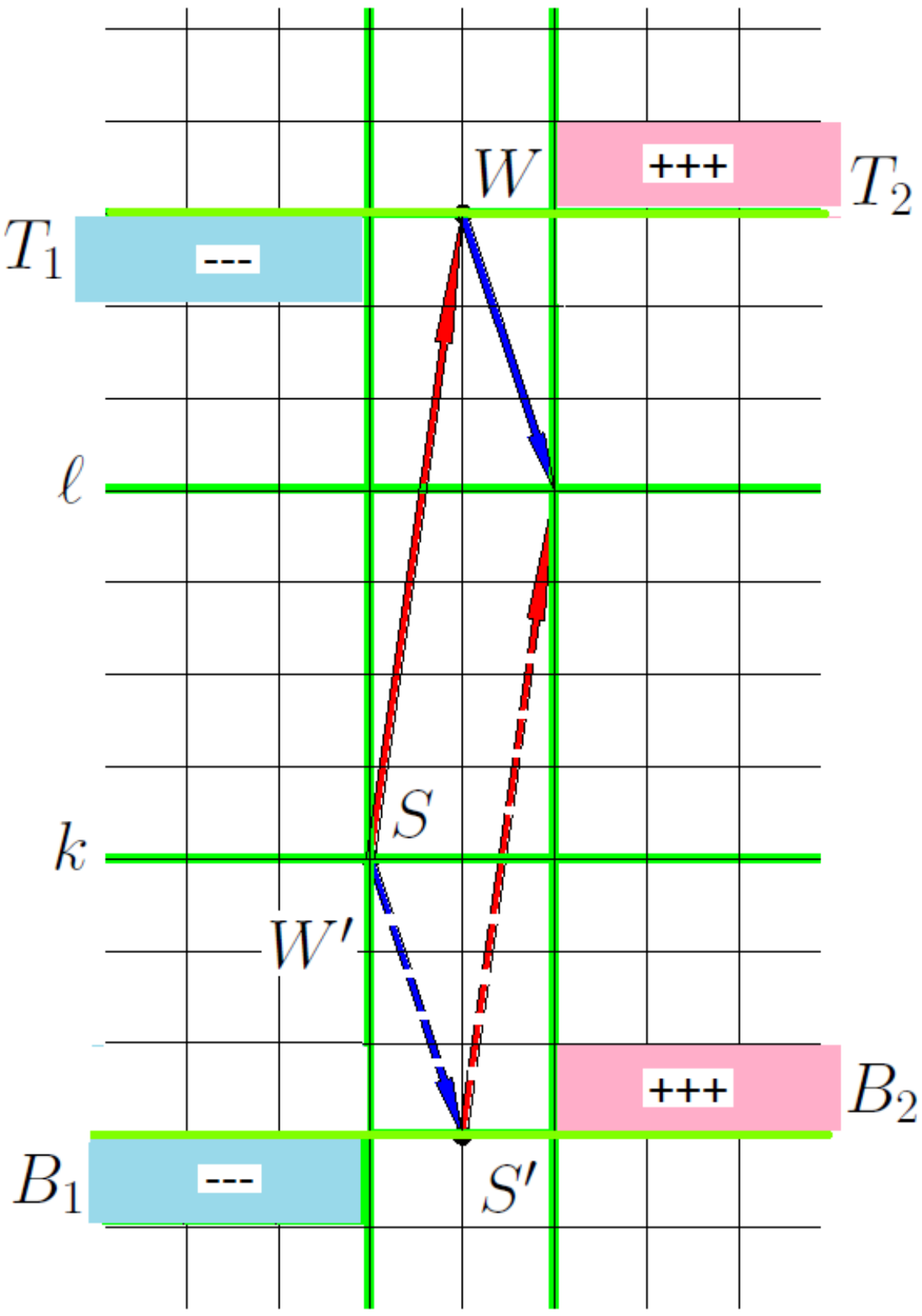}}{Contribution for the dinv difference. \label{fig:RedBlueA}}

Since $\overline D'$ is obtained from $\overline D$ by replacing the solid arrows $S,W$ by dashed arrows $S',W'$ (see Figure \ref{fig:RedBlueA}), we can divide the contribution of a pair  $(W_i\rightarrow S_j)$ to  the difference $\dinv(\overline D)-\dinv(\overline{D'})$ into four  cases.

\item {(1)} Both $W_i$ and $S_j$ are not the displayed arrows. The contribution in  this case is always $0$.

\item {(2)} Both $W_i$ and $S_j$ are in the displayed arrows. This can  only happen in two ways:  $(W,S)$ in $\overline D$ (no dinv) or  $(W',S')$ in $\overline D'$ ($1$  dinv). Therefore the
contribution to the difference in this case is $-1$.

\item {(3)} Only $W_i$ is one of  the displayed arrows. Then we need to consider
 the pairs $(W,S_j)$ in $\overline D$ and
$(W',S_j)$ in $\overline D'$. Their contribution to
the difference is $1$ if $S_j$ has a red segment in $T_2$, $-1$ if $S_j$ has a red segment in $B_2$ and $0$ if $S_j$ does not have a segment in neither $T_2$ or $B_2$.

\item {(4)} Only $S_j$ is in the displayed arrows. Then we need to consider  the pairs $(W_i,S )$ in $\overline D$ and  $(W_i,S')$ in $\overline D'$. Their contribution to the difference is $1$ if $W_i$ has a blue segment in $T_1$, $-1$ if $W_i$ has a blue segment in $B_1$ and $0$ if $W_i$ does not have a segment in neither $T_1$ or $B_1$.

\noindent
This proves the identity in \eqref{e-1.4}.
\end{proof}

\begin{proof}[Proof that dinv sweeps to area]
We first show that  the recursions in
\eqref{e-1.3} and \eqref{e-1.4} are identical. To do this it is sufficient to observe that
$c^r(T_1)=c^b(T_1)$  and  $c^r(B_2)+1=c^b(B_2)$. The reason for this is the alternating colors of segments in each row that always begin with a red segment and end with a blue segment.

Thus it is sufficient to verify the base case where $\area(\overline D)=0$. The only $(dm,dn)$-Dyck path with area $0$ is the path $\overline  D$ that remains  as close as possible to the main diagonal.  Thus the ranks of the South ends of such a path are a rearrangement of $0,1,2,\ldots,n-1$, with each appearing $d$ times.
This forces the image $D$ of  $\overline  D$ to start with $dn$ North steps  and end with $dm$ East steps.
 This is the path of maximum area. It remains to prove that $\overline  D$ has maximum dinv, or equivalently   every cell above $\overline  D$ contributes  to its dinv. By contradiction, suppose that for a pair $(W_i\to S_j)$,  the $i$-th blue arrow $W_i$ does not sweep the $j$-th red arrow $S_j$. We have the following two cases. See Figure \ref{fig:BaseCase}.

{(1)}  $S_j$ starts at a level  above $W_i$. Assume  we have
$S_j S_{j+1}\cdots S_{j+r} W$  for some $r\ge 0$. Then consider the path obtained from $\overline D$ by changing this to
$S_j S_{j+1}\cdots W S_{j+r} $. This is a new Dyck path with area one less than that of $\overline D$. A contradiction!

{(2)}  $S_j$ ends  below the level  of $W_i$, then assume we have $SW_{i-r} \cdots W_{i-1}W_i$ for some $r\ge 0$. Then consider the path obtained from $\overline D$ by changing this to $W_{i-r}S \cdots W_{i-1}W_i$. This is a new Dyck path with area one less than that of $\overline D$. A contradiction!

\begin{figure}[!ht]
\centering{
\mbox{\includegraphics[height=2 in]{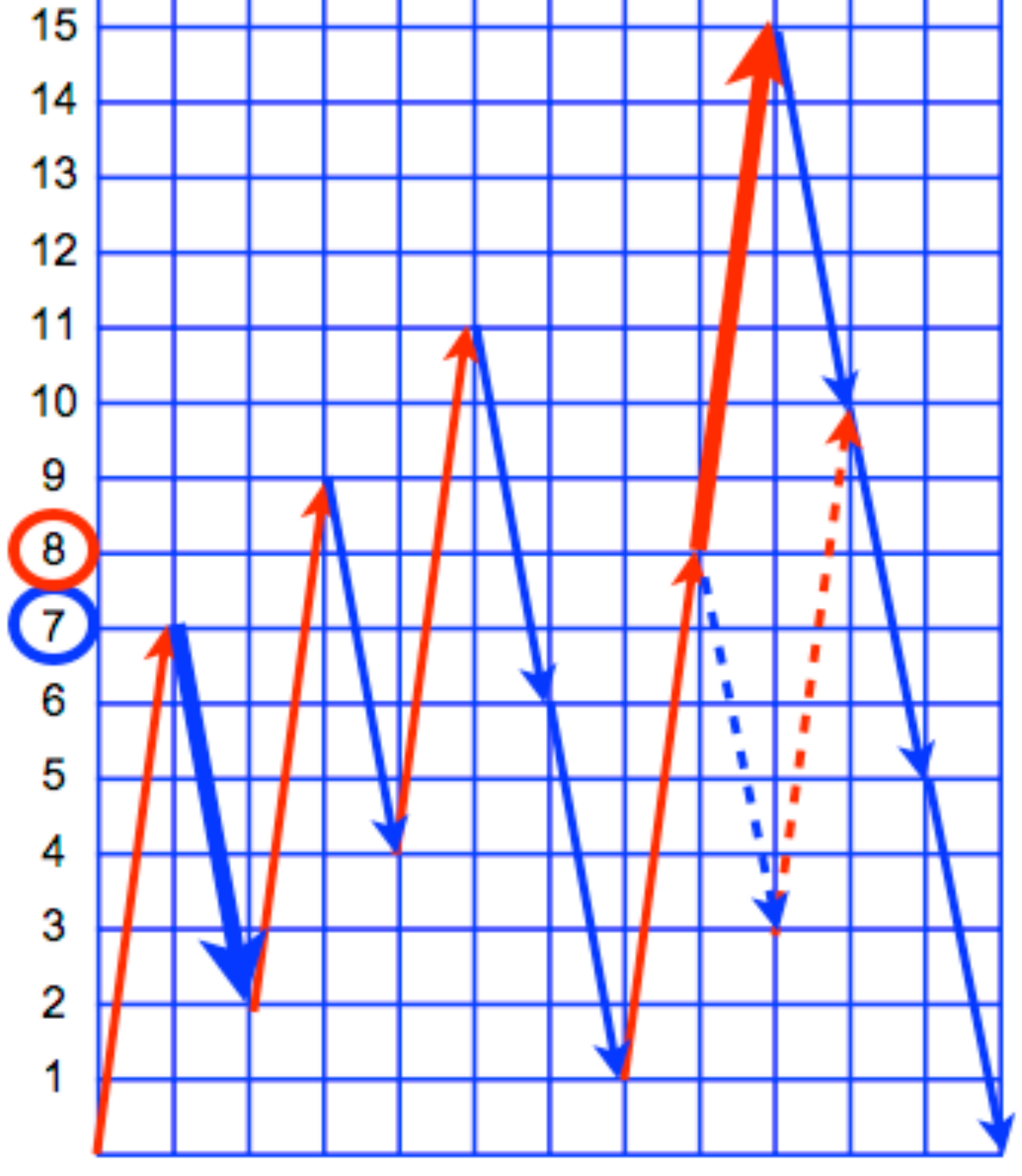}
\hspace{10mm}
\includegraphics[height=2 in]{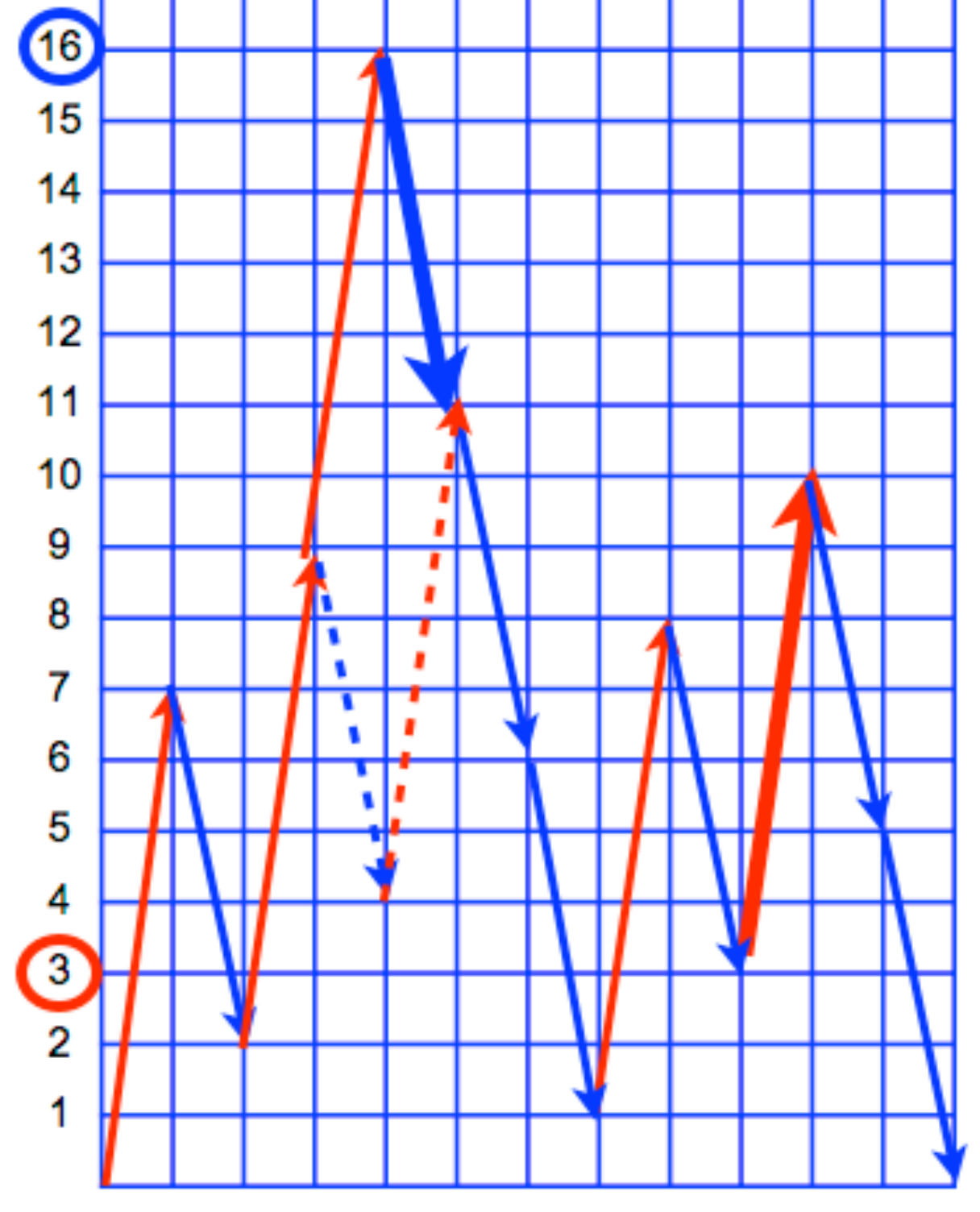}
}
\caption{Two examples of                                                                                                                                                                                                                                                                                                our proof that the dinv of the area zero $(m,n)$-Dyck path is $\frac{(dm-1)(dn-1)+d-1}{2}$, where the blue arrow $W_i$ and red arrow $S_j$ are thickened.}
\label{fig:BaseCase}
}
\end{figure}

This completes our proof that dinv sweeps to area.
\end{proof}

\section{Remarks}

We terminate our presentation by a few comments. To begin we should note that our argument does not use Proposition \ref{p-1.1}. We have nevertheless included it in this writing for two reasons.
Firstly because it is too surprising a result to leave out, but more importantly, because it gives a simple proof of the nontrivial result that the sweep image of a Dyck path is also a Dyck path. The reason for this is that it is implicit in the conclusion of the Proposition that  the starting ranks of all the arrows of the image are non-negative.
The latter is the only property needed to guarantee that the image of a $(dm,dn)$-Dyck path is  a $(dm,dn)$-Dyck path.

There are three proof of the ``dinv sweeps to area" result as we said in Section 1.
The first two proof only deal with the co-prime case. This case simplifies a lot since all starting ranks of the steps are distinct. Gregory Warrington told us their proof in \cite{ref1} can be extended for the non-coprime case. Our first draft of this paper also deal with the co-prime case, but explained how to (naturally) extend our approach to the non-coprime case. After we put this draft on the arXiv, Mazin told us immediately that the non-coprime case is  \cite[Corollary 1]{Mazin dinv sweep} after some translation of terminology. This write up is modified (suggested by the referee) to fit the non-coprime case.


\end{document}